\documentclass[11pt,twoside]{article}
\usepackage[utf8]{inputenc}
\usepackage[english]{babel}
\usepackage{fancyhdr}

\usepackage{amsmath,amssymb,amsfonts,amsthm}
\usepackage{mathtools}
\usepackage{algorithm}
\usepackage{algorithmicx}
\usepackage{algpseudocode}
\usepackage{multirow}
\usepackage{makecell}
\usepackage{bigstrut}
\usepackage{url}
\usepackage{xcolor}
\usepackage{hyperref}
\usepackage{caption}
\usepackage{subfig}
\usepackage{rotating}
\usepackage[title]{appendix}
\usepackage[square,numbers]{natbib}

\newtheorem{theorem}{Theorem}
\newtheorem{proposition}[theorem]{Proposition}
\newtheorem{lemma}[theorem]{Lemma}
\newtheorem{remark}{Remark}
\newtheorem{assumption}{Assumption}

\DeclareMathOperator*{\argmax}{Argmax}
\DeclarePairedDelimiter{\norm}{\lVert}{\rVert}
\DeclarePairedDelimiter{\abs}{\lvert}{\rvert}
\DeclarePairedDelimiter{\bigO}{\mathcal O(}{)}

\def\R{\ensuremath{\mathbb R}}
\def\Rn{\ensuremath{\mathbb R^n}}
\def\I{\ensuremath{\mathcal I}}
\def\J{\ensuremath{\mathcal J}}
\def\RI{\ensuremath{\R^{\abs{\I}}}}
\def\RIk{\ensuremath{\R^{\abs{\I_k}}}}
\def\IBCN{IBCN}
\def\BCDFIRST{BCD1}
\def\BCDSECOND{BCD2}
\def\alphahat{\ensuremath{\min\biggl\{\dfrac{\beta}{\norm{\nabla^2_{\I_k} f(x_k)}}, \sqrt{\dfrac{3\beta}{\sigma_k \norm{\nabla_{\I_k} f(x_k)}}}\biggr\}}}
\def\S{\ensuremath{\mathcal S}}
\def\sigmamin{\ensuremath{\sigma^{\text{min}}}}
\def\sigmamax{\ensuremath{\sigma^{\text{max}}}}
\def\fmin{\ensuremath{f^{\text{min}}}}
\def\Hmax{\ensuremath{B}}
\def\Lmin{\ensuremath{L^{\text{min}}}}

\def\Kepsblock{\ensuremath{K_{\epsilon}^{\text b}}}
\def\Keps{\ensuremath{K_{\epsilon}}}
\def\LS{\ensuremath{\mathcal L^0}}

\hypersetup{
	colorlinks=false,
	urlcolor=black,
	pdfborder={0 0 0}
}

\newcommand{\email}[1]{E-mail: \href{mailto:#1}{\texttt{#1}}}

\begin{document}
\thispagestyle{plain}

\setcounter{page}{1}

{\centering
	{\LARGE \bfseries Block cubic Newton with greedy selection}

	\bigskip\bigskip
	Andrea Cristofari$^*$
	\bigskip
	
}

\begin{center}
	\small{\noindent$^*$Department of Civil Engineering and Computer Science Engineering \\
		University of Rome ``Tor Vergata'' \\
		Via del Politecnico, 1, 00133 Rome, Italy \\
		\email{andrea.cristofari@uniroma2.it} \\
	}
\end{center}

\bigskip\par\bigskip\par
\noindent \textbf{Abstract.}
A second-order block coordinate descent method is proposed for the unconstrained minimization of an objective function with a Lipschitz continuous Hessian.
At each iteration, a block of variables is selected by means of a greedy (Gauss-Southwell) rule which considers the amount of first-order stationarity violation,
then an approximate minimizer of a cubic model is computed for the block update.
In the proposed scheme, blocks are not required to have a predetermined structure and their size may change during the iterations.
For non-convex objective functions, global convergence to stationary points is proved and a worst-case iteration complexity analysis is provided.
In particular, given a tolerance $\epsilon$,
we show that at most $\bigO{\epsilon^{-3/2}}$ iterations are needed to drive the stationarity violation with respect to a selected block of variables below $\epsilon$,
while at most $\bigO{\epsilon^{-2}}$ iterations are needed to drive the stationarity violation with respect to all variables below $\epsilon$.
Numerical results are finally given, comparing the proposed approach with other second-order methods and block selection rules.



\section{Introduction}\label{sec:intro}
Many challenging problems require the minimization of an objective function with several variables.
In this respect, block coordinate descent methods often represent an advantageous approach, especially when the objective function has a nice structure,
since these methods update a block of variables at each iteration and may have a low per-iteration cost.
In the literature, block coordinate descent methods have been extensively analyzed in several forms,
employing different rules to choose and update the blocks.

In particular, we can recognize three main block selection rules~\cite{nutini:2022,wright:2015}:
besides \textit{random} rules which choose blocks randomly, we have \textit{cyclic} (Gauss-Seidel) rules, which choose blocks by requiring
that each variable is selected at least once within any window of a pre-specified number of iterations, and \textit{greedy} (Gauss-Southwell) rules, which select the best block in terms of first-order stationarity violation at each iteration.
It is well known that greedy rules, compared to cyclic and random rules,
on the one hand require an additional effort to compute the whole gradient of the objective function,
but on the other hand tend to make a major progress per iteration~\cite{nutini:2022}.
Generally speaking, which block selection rule is the best in practice strongly depends on the features of the problems.
For first-order methods, theoretical and numerical comparisons of the different approaches can be found in, e.g., \cite{dhillon:2011,nutini:2015,venturini:2023},
showing that a greedy rule can be more efficient for some classes of problems with a certain structure.
In more detail, as discussed in~\cite{nutini:2015} for first-order coordinate descent methods, cyclic and random rules work well when the cost of performing $n$ coordinate updates is similar to the cost of performing one full gradient iteration, while a greedy rule can be efficiently used in other cases, e.g., when the objective function involves the multiplication of a sparse matrix by a vector.

Most block coordinate descent methods use first-order information and gained great popularity as they guarantee high efficiency in several applications.
However, when the objective function is twice continuously differentiable, second-order information can be conveniently used as well,
in order to speed up the convergence of the algorithm and overcome some drawbacks connected with first-order methods,
such as the performance deterioration in ill-conditioned or highly non-separable problems~\cite{fountoulakis:2018}.
Of course, second-order information should be used judiciously in a block coordinate descent scheme, so as not to increase the per-iteration cost excessively.
A possibility is extending, to a block coordinate descent setting, ideas from \textit{cubic Newton methods}~\cite{cartis:2011a,cartis:2011b,dussault:2018,dussault:2023,grapiglia:2017,griewank:1981,nesterov:2006},
where, at each iteration, the next point is obtained by minimizing a cubic model, that is,
a second-order model with cubic regularization (higher order models may also be considered when the objective function is several times continuously differentiable~\cite{birgin:2017,cartis:2019}).

In recent years, block coordinate descent versions of cubic Newton methods
were proposed in the literature using different block selection rules.
In particular, cyclic-type block selection was considered in~\cite{amaral:2022} for high order models which include cubic models as a special case,
whereas random block selection was analyzed in~\cite{doikov:2018,hanzely:2020} and~\cite{zhao:2024} for convex and non-convex objective functions, respectively.

To the best of the author's knowledge, greedy rules have not been investigated yet for
block coordinate descent versions of cubic Newton methods. The present paper aims to fill the gap by giving a detailed analysis of this scheme from both a theoretical and a numerical perspective.
Let us also highlight that, when using blocks made of $q$ variables, the number of second-order partial derivatives is of the order of $q^2$, so the cost of computing the full gradient of the objective function may not be dominant.
Hence, in a second-order setting, the computational disadvantage of a greedy rule over random and cyclic rules may reduce in some cases. However, as highlighted above, this is strongly related to the structure of the problem.

\subsection{Main contributions}
For the proposed block cubic Newton method with greedy selection, we provide the following worst-case iteration complexity bounds
to minimize a non-convex objective function with Lipschitz continuous Hessian:
\begin{itemize}
\item at most $\bigO{\epsilon^{-3/2}}$ iterations are needed to drive the stationarity violation with respect to \textit{a selected block of variables} below $\epsilon$,
\item at most $\bigO{\epsilon^{-2}}$ iterations are needed to drive the stationarity violation with respect to \textit{all variables} below $\epsilon$.
\end{itemize}

Our results are appealing if compared to those given in~\cite{amaral:2022} for cyclic-type block selection when using cubic models.
Specifically, the former complexity bound of $\bigO{\epsilon^{-3/2}}$ over a selected block of variables was obtained in~\cite{amaral:2022} as well,
but note that the latter complexity bound of $\bigO{\epsilon^{-2}}$ over all variables improves upon the one given in~\cite{amaral:2022}, which is $\bigO{\epsilon^{-3}}$.

Note also that first-order block coordinate descent methods need at most $\bigO{\epsilon^{-2}}$ iterations to reduce the stationarity violation over all variables below $\epsilon$ using either cyclic~\cite{beck:2013} or greedy~\cite{nutini:2015} block selection,
since they both guarantee an objective decrease proportional to the squared norm of the gradient of the objective function~\cite{beck:2013,nutini:2015}.

These results for cyclic and greedy selection in first- and second-order settings are summarized in Table~\ref{tab:literature}.

\begin{table}
\centering
\caption{Worst-case iteration complexity of cyclic and greedy selection when using gradient-type and cubic Newton updates on non-convex objective functions.}
{\begin{tabular}{c c c c}
\hline
Block update & Block selection & Worst-case iteration complexity & Reference\bigstrut[t]\bigstrut[b] \\
\hline
\multirow{2}*{\makecell[c]{gradient \\ (1st order)}} & cyclic & $\bigO{\epsilon^{-2}}$ & \cite{beck:2013} \bigstrut[t]\bigstrut[b] \\ \cline{2-4}
& greedy & $\bigO{\epsilon^{-2}}$ & \cite{nutini:2015} \bigstrut[t]\bigstrut[b] \\
\hline
\multirow{4}*{\makecell[c]{cubic Newton \\ (2nd order)}} & \multirow{2}*{cyclic} & $\bigO{\epsilon^{-3/2}}$ for selected block & \cite{amaral:2022} \bigstrut[t] \\
& & $\bigO{\epsilon^{-3}}$ for all variables & \cite{amaral:2022} \bigstrut[b] \\ \cline{2-4}
& \multirow{2}*{greedy} & $\bigO{\epsilon^{-3/2}}$ for selected block & this paper \bigstrut[t] \\
& & $\bigO{\epsilon^{-2}}$ for all variables & this paper \bigstrut[b] \\
\hline
\end{tabular}}
\label{tab:literature}
\end{table}

Hence, when using cubic Newton methods in a block coordinate descent setting, the literature indicates that the worst-case iteration complexity of cyclic selection to reduce the stationarity violation over all variables below $\epsilon$ is worse than that of first-order methods, namely $\bigO{\epsilon^{-3}}$ versus $\bigO{\epsilon^{-2}}$. 
In this fashion, the proposed greedy selection recovers the $\bigO{\epsilon^{-2}}$ complexity of first-order methods, while preserving the same efficiency of the cyclic selection in reducing the stationarity violation over a chosen block below $\epsilon$ within $\bigO{\epsilon^{-3/2}}$ iterations.

Let us remark that, for the proposed method,
we do not need to know the Lipschitz constant of the Hessian of the objective function.
Moreover, we use inexact minimizers of the cubic model whose computation does not require additional evaluations of the objective function or its derivatives.
Due to the block structure and the use of inexact information,
we name our algorithm \textit{Inexact Block Cubic Newton} (\IBCN) method.

We also assess the practical performance of the proposed \IBCN\ method on non-convex and convex problems.
In particular, we first demonstrate its efficiency compared with other block updates,
and then show that the proposed scheme can outperform cyclic and random variants when their per-iteration cost is similar to that of the greedy selection.

Furthermore, the code is freely available at \url{https://github.com/acristofari/ibcn}.

The rest of the paper is organized as follows.
In Section~\ref{sec:preliminaries}, we introduce the problem and give preliminary results.
In Section~\ref{sec:alg}, we describe the proposed method.
In Section~\ref{sec:conv}, we carry out the convergence analysis and give worst-case iteration complexity bounds.
In Section~\ref{sec:num}, we show some numerical results.
Finally, we draw some conclusions in Section~\ref{sec:concl}.

\section{Preliminaries and notations}\label{sec:preliminaries}
We consider the following unconstrained optimization problem:
\begin{equation}\label{prob}
\min_{x \in \Rn} f(x),
\end{equation}
where $f \colon \Rn \to \R$ is a (possibly non-convex) objective function.
We assume that the Hessian matrix $\nabla^2 f(x)$ is Lipschitz continuous over $\Rn$ with constant $L>0$, that is,
\[
\norm{\nabla^2 f(x) - \nabla^2 f(y)} \le L\norm{x-y} \quad \forall x,y \in \Rn,
\]
where, here and in the rest of the paper, $\norm{v}$ is the Euclidean norm for any vector $v$,
whereas $\norm{A}$ is the norm induced by the vector Euclidean norm for any matrix $A$.
The sup-norm of a vector $v$ is indicated by $\norm v_{\infty}$.

Given $\I \subseteq \{1,\ldots,n\}$, we denote by $U_{\I} \in \R^{n \times \abs{\I}}$ the submatrix of the $n$-dimensional identity matrix
obtained by removing all columns with indices not belonging to $\I$.
Then, given $x \in \Rn$ and $\I \subseteq \{1,\ldots,n\}$, we use the following notation:
\begin{itemize}
\item $x_{\I} \in \RI$ is the subvector of $x$ with elements in $\I$, that is,
    \[
    x_{\I} = U_{\I}^T x;
    \]
\item $\nabla_{\I} f(x) \in \RI$ is the vector of first-order partial derivatives of $f$ with respect to $x_i$, $i \in \I$, that is,
    \begin{equation}\label{subvec_g}
    \nabla_{\I} f(x) = U_{\I}^T \nabla f(x);
    \end{equation}
\item $\nabla^2_{\I} f(x) \in \R^{\abs{\I} \times \abs{\I}}$ is the matrix of second-order partial derivatives of $f$ with respect to $x_i$, $i \in \I$, that is,
    \begin{equation}\label{submat_h}
    \nabla^2_{\I} f(x) = U_{\I}^T \nabla^2 f(x) U_{\I}.
    \end{equation}
\end{itemize}
For example, if $n=5$ and
\[
x = \begin{bmatrix} 3 \\ 1 \\ 4 \\ -2 \\ 0 \end{bmatrix}, \quad \nabla f(x) = \begin{bmatrix} 2 \\ -1 \\ 0 \\ -3 \\ 4\end{bmatrix}, \quad
\nabla^2 f(x) = \begin{bmatrix} -2 & &  3 & & -6 & &  0 & & -7 \\
                                 3 & &  1 & & -5 & &  4 & &  2 \\
                                -6 & & -5 & &  7 & & -3 & & -1 \\
                                 0 & &  4 & & -3 & &  5 & & -4 \\
                                -7 & &  2 & & -1 & & -4 & &  6 \end{bmatrix},
\]
using $\I = \{1,3,4\}$ we get
\[
U_{\I} = \begin{bmatrix} 1 & 0 & 0 \\ 0 & 0 & 0 \\ 0 & 1 & 0 \\ 0 & 0 & 1 \\ 0 & 0 & 0 \end{bmatrix}, \quad
x_{\I} = \begin{bmatrix} 3 \\ 4 \\ -2 \end{bmatrix}, \quad \nabla_{\I} f(x) = \begin{bmatrix} 2 \\ 0 \\ -3 \end{bmatrix}, \quad
\nabla^2_{\I} f(x) = \begin{bmatrix} -2 & & -6 & &  0 \\
                                -6 & &  7 & & -3 \\
                                 0 & & -3 & &  5 \end{bmatrix}.
\]

Note that, for any choice of $\I \subseteq \{1,\ldots,n\}$, we have
\begin{align}
\norm{U_{\I}} & = 1 \label{norm_U}, \\
\norm{U_{\I} v} & = \norm{v} \quad \forall v \in \RI \label{norm_v}.
\end{align}

Moreover, for any choice of $\I \subseteq \{1,\ldots,n\}$, we define the block Lipschitz constant $L_{\I}$ such that
\begin{equation}\label{hess_lips}
\norm{\nabla^2_{\I} f(x+U_{\I}s) - \nabla^2_{\I} f(x)} \le L_{\I} \norm s \quad \forall x \in \Rn, \, \forall s \in \RI.
\end{equation}
Note that
\begin{equation}\label{L_I}
L_{\I} \in (0,L]
\end{equation}
since, recalling~\eqref{submat_h}, we have
\[
\begin{split}
\norm{\nabla^2_{\I} f(x+U_{\I}s) - \nabla^2_{\I} f(x)} & = \norm{U_{\I}^T (\nabla^2 f(x+U_{\I}s) - \nabla^2 f(x))U_{\I}} \\
                                              & \le \norm{\nabla^2 f(x+U_{\I}s) - \nabla^2 f(x)} \norm{U_{\I}}^2 \\
                                              & = \norm{\nabla^2 f(x+U_{\I}s) - \nabla^2 f(x)} \le L \norm{U_{\I}s} = L \norm s,
\end{split}
\]
where~\eqref{norm_U} has been used in the second equality and~\eqref{norm_v} has been used in the last equality.

Let us also define
\[
\Lmin = \min_{\I \subseteq \{1,\ldots,n\}} L_{\I}.
\]
From~\eqref{L_I}, it follows that
\begin{equation}\label{L_bounds}
0 < \Lmin \le L_{\I} \le L \quad \forall \I \subseteq \{1,\ldots,n\}.
\end{equation}

Extending known results on functions with Lipschitz continuous Hessian~\cite{dennis:1996,nesterov:2006},
we can give the following proposition whose proof is reported in Appendix~\ref{append:cubic}.

\begin{proposition}\label{prop:lips_block}
Given a point $x \in \Rn$ and a block of variable indices $\I \subseteq \{1,\ldots,n\}$, for all $s \in \RI$ we have that
\begin{gather}
\norm{\nabla_{\I} f(x+U_{\I}s) - \nabla_{\I} f(x) - \nabla^2_{\I} f(x) s} \le \frac{L_{\I}}2 \norm s^2, \label{ineq_lips_block} \\
\abs[Big]{f(x+U_{\I}s) - f(x) - \nabla_{\I} f(x)^T s - \frac12 s^T \nabla^2_{\I} f(x) s} \le \frac{L_{\I}}6 \norm s^3. \label{ub_lips_block}
\end{gather}
\end{proposition}

\section{The Inexact Block Cubic Newton (\IBCN) method}\label{sec:alg}
In this section, we describe the proposed algorithm, named \textit{Inexact Block Cubic Newton} (\IBCN) method,
which is reported in Algorithm~\ref{alg:ibcn}.
As will be shown, it is a natural extension of classical cubic Newton methods to a block coordinate descent scheme.

\subsection{Block selection}\label{sec:ws}
At the beginning of iteration $k$, we choose a block of variable indices $\I_k \subseteq \{1,\ldots,n\}$ by means of a classical greedy (or Gauss-Southwell) strategy~\cite{nutini:2022}, that is, $\I_k$ must include variables providing a sufficiently large amount of first-order stationarity violation. The rule can be stated as follows.

\begin{itemize}
\item[]\textbf{Greedy selection rule}:
    There exists a real number $\theta \in (0,1]$ such that
    \begin{equation}\label{ggr}
    \norm{\nabla_{\I_k} f(x_k)} \ge \theta \norm{\nabla f(x_k)} \quad \forall k \ge 0.
    \end{equation}
\end{itemize}

Note that the above greedy selection rule does not require the variables to be a priori partitioned into a fixed number of blocks,
so that even the size of blocks may change during the iterations.

In the following two propositions, we describe two simple procedures to satisfy~\eqref{ggr}.
For any iteration $k$, the first one requires $\I_k$ to include
the variable corresponding to the largest component in absolute value of $\nabla f(x_k)$,
while the second one, given an arbitrary number of (possibly overlapping) blocks of variables
covering $\{1,\ldots,n\}$, requires to compute the norm of the subvectors of $\nabla f(x_k)$ with respect to each block
in order to choose $\I_k$ as the one yielding the largest norm.

\begin{proposition}\label{prop:ws1}
For every iteration $k$, let $\hat \imath_k \in \argmax_{i=1,\ldots,n} \abs{\nabla_i f(x_k)}$ and assume that $\hat \imath_k \in \I_k$.
Then,
\[
\norm{\nabla_{\I_k} f(x_k)} \ge (n+1-\abs{\I_k})^{-1/2} \norm{\nabla f(x_k)} \quad \forall k \ge 0.
\]
It follows that~\eqref{ggr} is satisfied with
\[
\theta = \Bigl(n+1-\min_{k \ge 0}{\abs{\I_k}}\Bigr)^{-1/2}.
\]
\end{proposition}

\begin{proof}
Fix any iteration $k$ and let
\[
\tilde \I_k = (\{1,\ldots,n\} \setminus \I_k) \cup \{\hat \imath_k\}.
\]
Recalling the definition of $\hat \imath_k$, we have that
\[
\norm{\nabla_{\I_k} f(x_k)}_{\infty} = \norm{\nabla_{\tilde \I_k} f(x_k)}_{\infty} = \norm{\nabla f(x_k)}_{\infty} = \abs{\nabla_{\hat \imath_k} f(x_k)}.
\]
Then, we can write
\[
\begin{split}
\norm{\nabla_{\I_k} f(x_k)}^2 & = (\nabla_{\hat \imath_k} f(x_k))^2 + \sum_{i \in \I_k \setminus \{\hat \imath_k\}}\nabla_i f(x_k)^2 \\
& = \norm{\nabla_{\tilde \I_k} f(x_k)}_{\infty}^2 + \sum_{i \in \{1,\ldots,n\} \setminus \tilde \I_k}\nabla_i f(x_k)^2 \\
& \ge \abs{\tilde \I_k}^{-1} \left(\norm{\nabla_{\tilde \I_k} f(x_k)}^2 + \sum_{i \in \{1,\ldots,n\} \setminus \tilde \I_k}\nabla_i f(x_k)^2\right) \\
& = \abs{\tilde \I_k}^{-1} \norm{\nabla f(x_k)}^2.
\end{split}
\]
Since $\abs{\tilde \I_k} = n + 1 - \abs{\I_k}$, then the desired result follows.
\end{proof}

\begin{proposition}\label{prop:ws2}
For every iteration $k$, let $\J_k^1,\ldots,\J_k^{N_k}$ be subsets of $\{1,\ldots,n\}$ such that $\bigcup\limits_{j=1}^{N_k} \J_k^j = \{1,\ldots,n\}$
and assume that $\I_k \in \argmax_{\I = \J_k^1,\ldots,\J_k^{N_k}} \norm{\nabla_{\I} f(x_k)}$.
Then,
\[
\norm{\nabla_{\I_k} f(x_k)} \ge N_k^{-1/2} \norm{\nabla f(x_k)} \quad \forall k \ge 0.
\]
It follows that~\eqref{ggr} is satisfied with
\[
\theta = \min_{k \ge 0}{N_k^{-1/2}}.
\]
\end{proposition}

\begin{proof}
Fix any iteration $k$.
Since $\bigcup\limits_{j=1}^{N_k} \J_k^j = \{1,\ldots,n\}$, we can write
\[
\norm{\nabla f(x_k)}^2 \le \sum_{j=1}^{N_k} \norm{\nabla_{\J_k^j} f(x_k)}^2 \le N_k \norm{\nabla_{\I_k} f(x_k)}^2,
\]
where the last inequality follows from how $\I_k$ is selected, thus leading to the desired result.
\end{proof}

\subsection{Cubic model}
At iteration $k$, in order to update the variables in $\I_k$, we search for a suitable $s_k \in \RIk$ to move from $x_k$ along $U_{\I_k} s_k$.
To this aim, we define the cubic model $m_k(s)$ as follows:
\begin{equation}\label{mk_block}
m_k(s) = q_k(s) + \frac{\sigma_k}6 \norm s^3,
\end{equation}
where $\sigma_k$ is a positive scalar, updated during the iterations, which should overestimate $L_{\I_k}$, while $q_k(s)$ is the following quadratic model:
\begin{equation}\label{qk_block}
q_k(s) = f(x_k) + \nabla_{\I_k} f(x_k)^T s + \frac12 s^T \nabla^2_{\I_k} f(x_k)s.
\end{equation}
For the sake of convenience, let us also report the gradient of $m_k(s)$ as follows:
\begin{equation}\label{grad_mk_block}
\nabla m_k(s) = \nabla_{\I_k} f(x_k) + \nabla^2_{\I_k} f(x_k) s + \frac{\sigma_k}2 \norm{s} s.
\end{equation}

\subsection{Approximate minimizers of the cubic model}
At iteration $k$, we compute $s_k$ as an approximate minimizer of the cubic model~\eqref{mk_block}.
In particular, assuming that $\norm{\nabla_{\I_k} f(x_k)} \ne 0$, we require $s_k$ to satisfy two conditions.
The first one is that the first-order stationarity violation
must be sufficiently small compared to $\norm{s_k}^2$, that is,
\begin{equation}\label{sk_gm_block}
\norm{\nabla m_k(s_k)} \le \tau \norm{s_k}^2,
\end{equation}
with a given $\tau \in [0,\infty)$.
The second requirement is that $m_k(s_k)$ must be sufficiently low, that is,
\begin{equation}\label{sk_m_block}
\begin{split}
& m_k(s_k) \le m_k(\hat s_k), \quad \text{where} \\
& \hat s_k = -\hat \alpha_k \nabla_{\I_k} f(x_k) \quad \text{and} \quad \hat \alpha_k = \alphahat,
\end{split}
\end{equation}
with a given $\beta \in (0,1)$, letting the first argument within the above minimum to be $+\infty$ when $\norm{\nabla^2_{\I_k} f(x_k)} = 0$.

We see that condition~\eqref{sk_gm_block} is a straightforward adaptation of those used in~\cite{amaral:2022,birgin:2017,cartis:2019},
while condition~\eqref{sk_m_block} is inspired by the classical Cauchy condition~\cite{cartis:2011a} which requires
$m_k(s_k) \le \min_{\alpha \ge 0} m_k(-\alpha \nabla_{\I_k} f(x_k))$.
In our case, $m_k(s_k)$ is compared to $m_k(\hat s_k)$, hence~\eqref{sk_m_block} is weaker than the Cauchy condition.

In more detail, the role of the two conditions above in the convergence analysis can be outlined as follows:
\begin{itemize}
\item Condition~\eqref{sk_gm_block} allows us to upper bound $\norm{\nabla_{\I_k} f(x_{k+1})}$ by a term proportional to $\norm{s_k}^2$ at each iteration $k$
(see Proposition~\ref{prop:g_ub} below), which is needed for achieving a worst-case iteration complexity of $\bigO{\epsilon^{-3/2}}$
to reduce the stationarity violation over a selected block of variables below $\epsilon$ (see Theorem~\ref{th:compl_ibcn} below).
\item Condition~\eqref{sk_m_block} allows us to lower bound the decrease of both the cubic and quadratic models by a term
depending on $\norm{\nabla_{\I_k} f(x_k)}^2$ at each iteration $k$ (see Proposition~\ref{prop:mk_decr} and Lemma~\ref{lemma:q_decr} below).
This implies a lower bound on the objective decrease of the same order at certain iterations (see Proposition~\ref{prop:f_decr} below),
leading to a worst-case iteration complexity of $\bigO{\epsilon^{-2}}$ to reduce the stationarity violation over all variables below $\epsilon$,
thanks to the greedy rule used for the block selection (see Theorem~\ref{th:compl_ibcn_ggr} below).
\end{itemize}

Note that we can compute a vector $s_k$ satisfying~\eqref{sk_gm_block}--\eqref{sk_m_block} in finite time
without the need of additional evaluations of $f$ or its derivatives in other points.
In particular, we can apply an algorithm to approximately minimize the cubic model~\eqref{mk_block}
(using, e.g., the methods analyzed in~\cite{bianconcini:2015,carmon:2019,gould:2020,nesterov:2022}).
In our experiments, for the inexact minimization of the cubic model~\eqref{mk_block},
we use a Barzilai-Borwein gradient method~\cite{raydan:1997}, which was observed to be effective in practice~\cite{bianconcini:2015}.

Also, observe that~\eqref{sk_gm_block}--\eqref{sk_m_block} are clearly satisfied if $s_k$ is a global minimizer of the cubic model~\eqref{mk_block}
(details on how to compute global minimizers of a cubic model can be found in~\cite{cartis:2011a,cristofari:2019b,nesterov:2006}).

\subsection{Block update}
Once $s_k$ has been computed at iteration $k$, we use standard updating rules inherited from trust-region methods to decide how to set $x_{k+1}$ and $\sigma_{k+1}$ for the next iteration
(see, e.g.,~\cite{birgin:2017,cartis:2011a}). Namely, we compute
\begin{equation}\label{rhok}
\rho_k = \frac{f(x_k) - f(x_k + U_{\I_k} s_k)}{q_k(0) - q_k(s_k)}
\end{equation}
(we will show in Lemma~\ref{lemma:q_decr} below that the denominator is positive whenever $\nabla_{\I_k} f(x_k) \ne 0$) and check if
$\rho_k$ is sufficiently large.
In more detail, using $0 < \eta_1 \le \eta_2 < 1$ and $0 < \gamma_1 \le 1 < \gamma_2 \le \gamma_3$, we set
\begin{equation}\label{x_upd}
x_{k+1} =
\begin{cases}
x_k + U_{\I_k} s_k \quad & \text{if $\rho_k \ge \eta_1$,} \\
x_k \quad & \text{otherwise,}
\end{cases}
\end{equation}
and
\begin{equation}\label{sigma_upd}
\sigma_{k+1} \in
\begin{cases}
[\max\{\sigmamin,\gamma_1 \sigma_k\},\sigma_k] \quad & \text{if $\rho_k \ge \eta_2$,} \\
[\sigma_k, \gamma_2 \sigma_k] \quad & \text{if $\rho_k \in [\eta_1,\eta_2)$,} \\
[\gamma_2 \sigma_k, \gamma_3 \sigma_k] \quad & \text{if $\rho_k < \eta_1$,}
\end{cases}
\end{equation}
with $\sigmamin \in (0,\infty)$.

We also say that $k$ is a \textit{successful} iteration if $x_{k+1} = x_k + U_{\I_k} s_k$, denoting by $\S$ the set of all successful iterations. Namely, recalling~\eqref{x_upd}, we have
\begin{equation}\label{S_iter}
\S := \{k \text{ such that $\rho_k \ge \eta_1$}\}.
\end{equation}

\begin{algorithm}
	\caption{Inexact Block Cubic Newton (\IBCN) method}\label{alg:ibcn}
	\begin{algorithmic}[1]
		\State Given $x_0 \in \Rn$, $\sigmamin > 0$, $\sigma_0 \in [\sigmamin,\infty)$, $0 < \eta_1 \le \eta_2 < 1$, $0 < \gamma_1 \le 1 < \gamma_2 \le \gamma_3$, $\tau \in [0,\infty)$ and $\beta \in (0,1)$
		\While{$\nabla f(x_k) \ne 0$}
		\State compute $\I_k \subseteq \{1,\ldots,n\}$ such that $\norm{\nabla_{\I_k} f(x_k)} \ge \theta \norm{\nabla f(x_k)}$
		\State compute $s_k$ such that
		\begin{gather*}
			\norm{\nabla m_k(s_k)} \le \tau \norm{s_k}^2 \quad \text{and} \quad m_k(s_k) \le m_k(\hat s_k), \\
			\text{where} \\
			\begin{split}
				& \hat s_k = -\hat \alpha_k \nabla_{\I_k} f(x_k), \\ & \hat \alpha_k = \alphahat
			\end{split}
		\end{gather*}
		\State compute
		$\rho_k = \dfrac{f(x_k) - f(x_k + U_{\I_k} s_k)}{q_k(0) - q_k(s_k)}$
		\State set $\displaystyle{x_{k+1} =
			\begin{cases}
				x_k + U_{\I_k} s_k \quad & \text{if $\rho_k \ge \eta_1$} \\
				x_k \quad & \text{otherwise}
		\end{cases}}$
		\State set $\displaystyle{\sigma_{k+1} \in
			\begin{cases}
				[\max\{\sigmamin,\gamma_1 \sigma_k\},\sigma_k] \quad & \text{if $\rho_k \ge \eta_2$} \\
				[\sigma_k, \gamma_2 \sigma_k] \quad & \text{if $\rho_k \in [\eta_1,\eta_2)$} \\
				[\gamma_2 \sigma_k, \gamma_3 \sigma_k] \quad & \text{if $\rho_k < \eta_1$}
		\end{cases}}$
		\EndWhile
	\end{algorithmic}
\end{algorithm}

\section{Convergence analysis}\label{sec:conv}
For the proposed \IBCN\ method, we first establish global convergence to stationary points in Subsection~\ref{subsec:glob_conv},
and then present worst-case iteration complexity results in Subsection~\ref{subsec:complexity}.

\subsection{Global convergence}\label{subsec:glob_conv}
To begin with, for each iteration we can provide a lower bound on the decrease of the cubic model similarly to when using the Cauchy step~\cite{cartis:2011a}.

\begin{proposition}\label{prop:mk_decr}
For every iteration $k$, we have
\[
m_k(0) - m_k(s_k) \ge m_k(0) - m_k(\hat s_k) \ge (1 - \beta) \hat \alpha_k \norm{\nabla_{\I_k} f(x_k)}^2,
\]
where $\hat s_k$ is defined as in~\eqref{sk_m_block}.
\end{proposition}

\begin{proof}
The first inequality of the thesis follows from~\eqref{sk_m_block},
so we only have to show the second inequality.
To this aim, recalling the definitions of $m_k$ and $\nabla m_k$ from~\eqref{mk_block} and~\eqref{grad_mk_block}, respectively, we can write
\[
\begin{split}
m_k(0) - m_k(\hat s_k) & = f(x_k) - m_k (-\hat \alpha_k \nabla_{\I_k} f(x_k)) \\
& = \hat \alpha_k \norm{\nabla_{\I_k} f(x_k)}^2 - \frac{\hat \alpha_k^2}2 \nabla_{\I_k} f(x_k)^T \nabla^2_{\I_k} f(x_k) \nabla_{\I_k} f(x_k) + \\
& \quad - \frac{\hat \alpha_k^3 \sigma_k}6 \norm{\nabla_{\I_k} f(x_k)}^3 \\
& \ge \hat \alpha_k \norm{\nabla_{\I_k} f(x_k)}^2 \Bigl(1 - \frac{\hat \alpha_k\norm{\nabla^2_{\I_k} f(x_k)}}2 - \frac{\hat \alpha_k^2 \sigma_k \norm{\nabla_{\I_k} f(x_k)}}6 \Bigr).
\end{split}
\]
From the definition of $\hat \alpha_k$ given in~\eqref{sk_m_block}, it follows that
\[
1 - \frac{\hat \alpha_k \norm{\nabla^2_{\I_k} f(x_k)}}2 - \frac{\hat \alpha_k^2 \sigma_k \norm{\nabla_{\I_k} f(x_k)}}6 \ge 1 - \frac{\beta}2 - \frac{\beta}2 = 1 - \beta.
\]
Then, the desired result follows.
\end{proof}

Using the above proposition, we can easily lower bound the decrease of the quadratic model at every iteration as follows.

\begin{lemma}\label{lemma:q_decr}
For every iteration $k$, we have
\[
q_k(0) - q_k(s_k) \ge (1 - \beta) \hat \alpha_k \norm{\nabla_{\I_k} f(x_k)}^2 + \frac{\sigma_k}6 \norm{s_k}^3.
\]
\end{lemma}

\begin{proof}
For any iteration $k$, from~\eqref{qk_block} and~\eqref{mk_block} it follows that
\[
q_k(0) - q_k(s_k) = m_k(0) - m_k(s_k) + \frac{\sigma_k}6 \norm{s_k}^3.
\]
Hence, the desired result is obtained by using Proposition~\ref{prop:mk_decr}.
\end{proof}

In the following two propositions we show how, for every successful iteration $k$,
we can lower bound $(f(x_k) - f(x_{k+1}))$ and upper bound $\norm{\nabla_{\I_k} f(x_{k+1})}$.

\begin{proposition}\label{prop:f_decr}
For every iteration $k \in \S$, we have
\[
f(x_k) - f(x_{k+1}) \ge \eta_1 \biggl((1 - \beta) \hat \alpha_k \norm{\nabla_{\I_k} f(x_k)}^2 + \frac{\sigma_k}6 \norm{s_k}^3\biggr).
\]
\end{proposition}

\begin{proof}
Take any $k \in \S$.
From the instructions of the algorithm and the definition of $\S$ given in~\eqref{S_iter}, we have that $\rho_k \ge \eta_1$ and $x_{k+1} = x_k + U_{\I_k} s_k$.
Recalling the definition of $\rho_k$ given in~\eqref{rhok},
then the desired result follows from Lemma~\ref{lemma:q_decr}.
\end{proof}

\begin{remark}\label{rem:f_decr}
According to the instructions of the algorithm and the definition of $\S$ given in~\eqref{S_iter}, we have that $x_{k+1}=x_k$ when $k \notin \S$. Then, using
Proposition~\ref{prop:f_decr}, it follows that the sequence $\{f(x_k)\}$ is monotonically non-increasing.
\end{remark}

\begin{proposition}\label{prop:g_ub}
For every iteration $k \in \S$, we have
\[
\norm{\nabla_{\I_k} f(x_{k+1})} \le \Bigl(\tau + \frac{\sigma_k+L_{\I_k}}2\Bigr) \norm{s_k}^2.
\]
\end{proposition}

\begin{proof}
Take any $k \in \S$. First, we can write
\begin{equation}\label{ineq_g_1}
\norm{\nabla_{\I_k} f(x_{k+1})} \le
\norm{\nabla_{\I_k} f(x_k) + \nabla^2_{\I_k} f(x_k) s_k} + \norm{\nabla_{\I_k} f(x_{k+1}) - \nabla_{\I_k} f(x_k) - \nabla^2_{\I_k} f(x_k) s_k}.
\end{equation}
Using~\eqref{grad_mk_block}, we can upper bound the first norm in the right-hand side of~\eqref{ineq_g_1} as follows:
\begin{equation}\label{ineq_g_2}
\begin{split}
\norm{\nabla_{\I_k} f(x_k) + \nabla^2_{\I_k} f(x_k) s_k} & =
\norm[\Big]{\nabla m_k(s_k) - \frac{\sigma_k}2 \norm{s_k}s_k} \\
& \le \norm{\nabla m_k(s_k)} + \frac{\sigma_k}2 \norm{s_k}^2 \\
& \le \Bigl(\tau + \frac{\sigma_k}2\Bigr) \norm{s_k}^2,
\end{split}
\end{equation}
where the last inequality follows from~\eqref{sk_gm_block}.
From the instructions of the algorithm and the definition of $\S$ given in~\eqref{S_iter}, we have that $x_{k+1} = x_k + U_{\I_k} s_k$.
So, using~\eqref{ineq_lips_block},
we can also upper bound the second norm in the right-hand side of~\eqref{ineq_g_1} as follows:
\begin{equation}\label{ineq_g_3}
\norm{\nabla_{\I_k} f(x_{k+1}) - \nabla_{\I_k} f(x_k) - \nabla^2_{\I_k} f(x_k) s_k} \le \frac{L_{\I_k}}2 \norm{s_k}^2.
\end{equation}
Then, the desired result follows from~\eqref{ineq_g_1}, \eqref{ineq_g_2} and~\eqref{ineq_g_3}.
\end{proof}

Now, we want to relate the total number of iterations to the successful ones,
which will be obtained in Proposition~\ref{prop:U_ub} below.
To get such a result, we have to pass through a few intermediate steps.
First we show that $\sigma_{k+1} \le \sigma_k$ whenever $\sigma_k$ is an appropriate overestimate of $L_{\I_k}$.

\begin{proposition}\label{prop:sigmak_S}
Assume that, for an iteration $k \ge 0$, we have
\[
\sigma_k \ge \frac{L_{\I_k}}{1-\eta_2}.
\]
Then, $\sigma_{k+1} \le \sigma_k$.
\end{proposition}

\begin{proof}
From~\eqref{qk_block}, we can write
\begin{equation}\label{ineq_m_decr_1}
-\nabla_{\I_k} f(x_k)^T s_k - \frac12 s_k^T \nabla^2_{\I_k} f(x_k) s_k = q_k(0) - q_k(s_k).
\end{equation}
Using Lemma~\ref{lemma:q_decr}, it follows that
\begin{equation}\label{ineq_m_decr_2}
-\nabla_{\I_k} f(x_k)^T s_k - \frac12 s_k^T \nabla^2_{\I_k} f(x_k) s_k \ge \frac{\sigma_k}6 \norm{s_k}^3.
\end{equation}
Using~\eqref{ineq_m_decr_1} and~\eqref{ub_lips_block}, from the definition of $\rho_k$ given in~\eqref{rhok} we obtain
\[
1- \rho_k = \frac{-\nabla_{\I_k} f(x_k)^T s_k - \frac12 s_k^T \nabla^2_{\I_k} f(x_k) s_k - f(x_k) + f(x_k + U_{\I_k} s_k)}
                 {-\nabla_{\I_k} f(x_k)^T s_k - \frac12 s_k^T \nabla^2_{\I_k} f(x_k) s_k} \le \frac{L_{\I_k}}{\sigma_k},
\]
where, in the last inequality, we have used~\eqref{ub_lips_block} and~\eqref{ineq_m_decr_2}
to upper bound the numerator by $(L_{\I_k}/6) \norm{s_k}^3$ and lower bound the denominator by $(\sigma_k/6) \norm{s_k}^3$, respectively.
Since $\sigma_k \ge L_{\I_k}/(1-\eta_2)$ by hypothesis, it follows that $1 - \rho_k \le 1 - \eta_2$,
that is, $\rho_k \ge \eta_2$. Then the desired result follows from the updating rule of $\sigma_k$ given in~\eqref{sigma_upd}.
\end{proof}

In the following proposition, we show that $\sigma_k$ stays bounded during the iterations.

\begin{proposition}\label{prop:sigma_bounds}
For every iteration $k\ge 0$, it holds that
\[
\sigmamin \le \sigma_k \le \sigmamax,
\]
where $\sigmamax := \max\biggl\{\sigma_0,\dfrac{\gamma_3 L}{1-\eta_2}\biggr\}$.
\end{proposition}

\begin{proof}
From Proposition~\ref{prop:sigmak_S} and the fact that, by the instructions of the algorithm, $\sigma_{k+1}$ is increased from $\sigma_k$
at most by a factor of $\gamma_3 > 1$, it follows that
\[
\sigma_{k+1} \le \max\biggl\{\sigma_k,\frac{\gamma_3 L_{\I_k}}{1-\eta_2}\biggr\} \quad \forall k \ge 0.
\]
Applying this inequality recursively and upper bounding $L_{\I_k}$ by~\eqref{L_bounds}, we get the desired upper bound on $\sigma_k$,
while the lower bound follows immediately from the instructions of the algorithm.
\end{proof}

We can now upper bound the total number of unsuccessful iterations up to the current iteration.

\begin{proposition}\label{prop:U_ub}
For every iteration $k\ge 0$, let $\S_k$ be the set of successful iterations up to $k$, that is, $\S_k := \S \cap \{i \le k\}$.
Then,
\[
k+1 \le \biggl(1 + \frac{\abs{\log \gamma_1}}{\log \gamma_2}\biggr)\abs{\S_k} + \frac 1{\log \gamma_2} \log\biggl(\frac{\sigmamax}{\sigma_0}\biggr)
\]
\end{proposition}

\begin{proof}
It is identical to~\cite[Lemma 2.4]{birgin:2017}, where the same rule is used to update $\sigma_k$ during the iterations.
\end{proof}

To give worst-case iteration complexity bounds, we need an assumption on the boundedness of $f$ and $\norm{\nabla^2 f}$ over the following level set:
\begin{equation}\label{LS}
\LS = \{x \in \Rn \colon f(x) \le f(x_0)\}.
\end{equation}

\begin{assumption}\label{assump:bound}
Two finite constants $\fmin$ and $\Hmax$ exist such that, for all $x \in \LS$, we have
$f(x) \ge \fmin$ and $\norm{\nabla^2 f(x)} \le \Hmax$, where $\LS$ is defined as in~\eqref{LS}.
\end{assumption}
We see that Assumption~\ref{assump:bound} is satisfied if $\LS$ is compact.
Note also that, since $\{f(x_k)\}$ is monotonically non-increasing from Remark~\ref{rem:f_decr}, then
$\{x_k\} \subseteq \LS$. It follows that, under Assumption~\ref{assump:bound}, we have
\begin{align}
f(x_k) \ge \fmin & \quad \forall k \ge 0, \label{fmin_k} \\
\norm{\nabla^2_{\I_k} f(x_k)} \le \Hmax & \quad \forall k \ge 0. \label{H_bounded_k}
\end{align}

Under Assumption~\ref{assump:bound}, we can now relate the objective decrease at successful iterations to the amount of first-order stationarity violation.

\begin{proposition}\label{prop:f_decr_eps}
Given $\epsilon \in [0,1]$, if Assumption~\ref{assump:bound} holds, then
\[
f(x_k) - f(x_{k+1}) \ge c_1 \epsilon^2 \quad \forall k \in \S \colon \norm{ \nabla f(x_k)} \ge \epsilon,
\]
where
\[
c_1 = \theta \eta_1 (1 - \beta) \min\biggl\{\frac{\theta \beta}\Hmax, \sqrt{\dfrac{3 \theta \beta}{\sigmamax}}\biggr\}
\]
and $\sigmamax$ is given in Proposition~\ref{prop:sigma_bounds}.
\end{proposition}

\begin{proof}
Take any iteration $k \in \S$ such that $\norm{ \nabla f(x_k)} \ge \epsilon$, with $\epsilon \in [0,1]$. From the instructions of the algorithm and the definition of $\S$ given in~\eqref{S_iter}, we have that $x_{k+1} = x_k + U_{\I_k} s_k$.
Using Proposition~\ref{prop:f_decr} and the greedy selection rule~\eqref{ggr}, we have that
\[
f(x_k) - f(x_{k+1}) \ge \eta_1 (1 - \beta) \hat \alpha_k \norm{\nabla_{\I_k} f(x_k)}^2
\ge \theta \eta_1 (1 - \beta) \hat \alpha_k \norm{\nabla f(x_k)} \norm{\nabla_{\I_k} f(x_k)}.
\]
Since $\norm{ \nabla f(x_k)} \ge \epsilon$, we obtain
\begin{equation}\label{f_decr_compl_GR_2}
f(x_k) - f(x_{k+1}) \ge \theta \eta_1 (1 - \beta) \hat \alpha_k \epsilon \norm{\nabla_{\I_k} f(x_k)}.
\end{equation}
Now, using the definition of $\hat \alpha_k$ given in~\eqref{sk_m_block}, we can write
\[
\hat \alpha_k \norm{\nabla_{\I_k} f(x_k)} =
\min\biggl\{\frac{\beta \norm{\nabla_{\I_k} f(x_k)}}{\norm{\nabla^2_{\I_k} f(x_k)}}, \sqrt{\dfrac{3 \beta \norm{\nabla_{\I_k} f(x_k)}}{\sigma_k}}\biggr\}.
\]
Since, from~\eqref{H_bounded_k} and Proposition~\ref{prop:sigma_bounds}, respectively, $\norm{\nabla^2_{\I_k} f(x_k)} \le \Hmax$ and $\sigma_k \le \sigmamax$, we get
\[
\hat \alpha_k \norm{\nabla_{\I_k} f(x_k)} \ge \min\biggl\{\frac{\beta \norm{\nabla_{\I_k} f(x_k)}}{\Hmax}, \sqrt{\dfrac{3 \beta \norm{\nabla_{\I_k} f(x_k)}}{\sigmamax}}\biggr\}.
\]
So, using the greedy selection rule~\eqref{ggr} and the fact that $\norm{ \nabla f(x_k)} \ge \epsilon$, with $\epsilon \in [0,1]$, we obtain
\begin{equation}\label{alpha_hat_lb}
\hat \alpha_k \norm{\nabla_{\I_k} f(x_k)} \ge \min\biggl\{\frac{\theta \beta \epsilon}\Hmax, \sqrt{\dfrac{3 \theta \beta \epsilon}{\sigmamax}}\biggr\}
\ge \epsilon \min\biggl\{\frac{\theta \beta}\Hmax, \sqrt{\dfrac{3 \theta \beta}{\sigmamax}}\biggr\}.
\end{equation}
Then, combining~\eqref{f_decr_compl_GR_2} and~\eqref{alpha_hat_lb}, the desired result follows.
\end{proof}

Using the above result, we can easily show the global convergence of \IBCN\ to stationary points.

\begin{theorem}\label{th:conv}
If Assumption~\ref{assump:bound} holds, then
\[
\lim_{k \to \infty} \nabla f(x_k) = 0.
\]
\end{theorem}

\begin{proof}
Since, in view of the definition of $\S$ given in~\eqref{S_iter}, we have $\nabla f(x_{k+1}) = \nabla f(x_k)$ for all $k \notin \S$, then the result is true if and only if
\[
\lim_{\substack{k \to \infty \\ k \in \S}} \nabla f(x_k) = 0.
\]
By contradiction, assume that $\epsilon \in (0,1]$ and an infinite subset $K \subseteq \S$ exist such that $\norm{\nabla f(x_k)} \ge \epsilon$ for all $k \in K$.
From Proposition~\ref{prop:f_decr_eps}, it follows that $f$ is unbounded from below,
thus contradicting~\eqref{fmin_k}.
\end{proof}

\begin{remark}
From Theorem~\ref{th:conv} and the continuity of $\nabla f$, it follows that every limit point of $\{x_k\}$ is a stationary point.
\end{remark}

\subsection{Worst-case iteration complexity}\label{subsec:complexity}
Here, we analyze the worst-case iteration complexity of the proposed \IBCN\ method, providing two main results.

First, in the following theorem, we show that at most $\bigO{\epsilon^{-3/2}}$ iterations are needed
to drive $\norm{\nabla_{\I_k}f(x_{k+1})}$ below a given threshold $\epsilon > 0$.
Note that, in the proof of the following theorem, no role is played by the greedy selection rule~\eqref{ggr}, that is,
the result holds for any arbitrary choice of the blocks.

\begin{theorem}\label{th:compl_ibcn}
Given $\epsilon > 0$, let
\[
\Kepsblock = \{k \ge 0 \colon \norm{\nabla_{\I_k}f(x_{k+1})} \ge \epsilon\}.
\]
If Assumption~\ref{assump:bound} holds, then
\[
\abs{\Kepsblock \cap \S} \le \biggl(\frac{f(x_0)-\fmin}{c_2}\biggr) \epsilon^{-3/2},
\]
where
\[
c_2 = \frac{\eta_1 \sigmamin}6 \biggl(\tau + \frac{\sigmamax+L}2\biggr)^{-3/2}.
\]
Moreover, let $\bar k$ be the first iteration such that $\bar k \notin \Kepsblock$. Then,
\[
\bar k \le \biggl(1 + \frac{\abs{\log \gamma_1}}{\log \gamma_2}\biggr) \biggl(\frac{f(x_0)-\fmin}{c_2}\biggr) \epsilon^{-3/2} + \frac 1{\log \gamma_2} \log\biggl(\frac{\sigmamax}{\sigma_0}\biggr),
\]
where $\sigmamax$ is given in Proposition~\ref{prop:sigma_bounds}.
\end{theorem}

\begin{proof}
From Proposition~\ref{prop:f_decr} and the definition of $\S$ given in~\eqref{S_iter}, we have that $f(x_k) - f(x_{k+1}) \ge 0$ for all $k \ge 0$.
So, using the lower bound on $f(x)$ given in~\eqref{fmin_k}, we can write
\[
f(x_0) - \fmin \ge f(x_0)-f(x_{r+1}) \ge \sum_{\substack{h=0 \\ h \in \Kepsblock \cap \S}}^r (f(x_h) - f(x_{h+1})) \quad \forall r \ge 0.
\]
Taking the limit for $r \to \infty$, we get
\begin{equation}\label{f_decr_total_compl}
f(x_0) - \fmin \ge \sum_{k \in \Kepsblock \cap \S} (f(x_k) - f(x_{k+1})).
\end{equation}
Now we want to lower bound the right-hand side term of~\eqref{f_decr_total_compl}.
First, from Proposition~\ref{prop:f_decr}, we have that
\begin{equation}\label{f_decr_compl_1}
f(x_k) - f(x_{k+1}) \ge \frac{\eta_1 \sigma_k}6 \norm{s_k}^3 \ge \frac{\eta_1 \sigmamin}6 \norm{s_k}^3 \quad \forall k \in \S,
\end{equation}
where, in the last inequality, we have used Proposition~\ref{prop:sigma_bounds} to lower bound $\sigma_k$.
Moreover, from Proposition~\ref{prop:g_ub}, we have that
\begin{equation}\label{g_ub_compl}
\norm{\nabla_{\I_k} f(x_{k+1})} \le \biggl(\tau + \frac{\sigma_k+L_{\I_k}}2\biggr) \norm{s_k}^2 \le \biggl(\tau + \frac{\sigmamax+L}2\biggr) \norm{s_k}^2 \quad \forall k \in \S,
\end{equation}
where, in the last inequality, we have used Proposition~\ref{prop:sigma_bounds} and~\eqref{L_bounds}
to upper bound $\sigma_k$ and $L_{\I_k}$, respectively.
Therefore, from~\eqref{f_decr_compl_1} and~\eqref{g_ub_compl}, we obtain
\begin{equation}\label{f_decr_compl_2}
f(x_k) - f(x_{k+1}) \ge c_2 \norm{\nabla_{\I_k} f(x_{k+1})}^{3/2} \quad \forall k \in \S.
\end{equation}
It follows that
\[
f(x_k) - f(x_{k+1}) \ge c_2 \epsilon^{3/2} \quad \forall k \in \S \text{ such that } \norm{\nabla_{\I_k} f(x_{k+1})} \ge \epsilon.
\]
Using this inequality in~\eqref{f_decr_total_compl}, we obtain
\[
f(x_0) - \fmin \ge \abs{\Kepsblock \cap \S} c_2 \epsilon^{3/2},
\]
thus leading to the desired upper bound on $\abs{\Kepsblock \cap \S}$.
To obtain the desired upper bound on $\bar k$, it is trivial if $\bar k = 0$,
otherwise the result follows by using Proposition~\ref{prop:U_ub} with $k = \bar k - 1$ and noting that $\abs{\S_{\bar k - 1}} \le \abs{\Kepsblock \cap \S}$.
\end{proof}

From Theorem~\ref{th:compl_ibcn} we see that, to drive the stationarity violation with respect to \textit{a selected block of variables} below $\epsilon$,
we need at most $\bigO{\epsilon^{-3/2}}$ iterations, thus matching the complexity bound given in~\cite{amaral:2022} for cyclic-type selection.

Moreover, when $\I_k = \{1,\ldots,n\}$ for all $k$, we retain the complexity bound of standard cubic Newton methods, that is,
at most $\bigO{\epsilon^{-3/2}}$ iterations are needed to obtain $\norm{\nabla f(x_k)} < \epsilon$.

In a general case where $\abs{\I_k} < n$, Theorem~\ref{th:compl_ibcn} does not provide information on how many iterations
are needed in the worst case to drive the stationarity violation with respect to \textit{all variables} below $\epsilon$.
Such a complexity bound is given in the next theorem, ensuring that at most $\bigO{\epsilon^{-2}}$ iterations are needed to get $\norm{\nabla f(x_k)} < \epsilon$.

\begin{theorem}\label{th:compl_ibcn_ggr}
Given $\epsilon \in (0,1]$, let
\[
\Keps = \{k \ge 0 \colon \norm{\nabla f(x_k)} \ge \epsilon\}.
\]
If Assumption~\ref{assump:bound} holds, then
\[
\abs{\Keps \cap \S} \le \biggl(\frac{f(x_0)-\fmin}{c_1}\biggr) \epsilon^{-2},
\]
where $c_1$ is defined as in Proposition~\ref{prop:f_decr_eps}.
Moreover, let $\hat k$ be the first iteration such that $\hat k \notin \Keps$. Then,
\[
\hat k \le \biggl(1 + \frac{\abs{\log \gamma_1}}{\log \gamma_2}\biggr) \biggl(\frac{f(x_0)-\fmin}{c_1}\biggr) \epsilon^{-2} + \frac 1{\log \gamma_2} \log\biggl(\frac{\sigmamax}{\sigma_0}\biggr),
\]
where $\sigmamax$ is given in Proposition~\ref{prop:sigma_bounds}.
\end{theorem}

\begin{proof}
From Proposition~\ref{prop:f_decr} and the definition of $\S$ given in~\eqref{S_iter}, we have that $f(x_k) - f(x_{k+1}) \ge 0$ for all $k \ge 0$.
So, using the lower bound on $f(x)$ given in~\eqref{fmin_k}, we can write
\[
f(x_0) - \fmin \ge f(x_0)-f(x_{r+1}) \ge \sum_{\substack{h=0 \\ h \in \Keps \cap \S}}^r (f(x_h) - f(x_{h+1})) \quad \forall r \ge 0.
\]
Taking the limit for $r \to \infty$, we get
\[
f(x_0) - \fmin \ge \sum_{k \in \Keps \cap \S} (f(x_k) - f(x_{k+1})).
\]
From Proposition~\ref{prop:f_decr_eps}, it follows that
\[
f(x_0) - \fmin \ge \abs{\Keps \cap \S} c_1 \epsilon^2,
\]
thus leading to the desired upper bound on $\abs{\Keps \cap \S}$.
To obtain the desired upper bound on $\hat k$, it is trivial if $\hat k = 0$,
otherwise the result follows by using Proposition~\ref{prop:U_ub} with $k = \hat k - 1$ and noting that $\abs{\S_{\hat k - 1}} \le \abs{\Keps \cap \S}$.
\end{proof}

\begin{remark}\label{rem:theta}
Since $c_1 = \bigO{\theta^{3/2}}$, with $\theta \in (0,1]$, it follows that the larger $\theta$ the better the complexity bound of Theorem~\ref{th:compl_ibcn_ggr}.
Values for $\theta$ have been derived in Propositions~\ref{prop:ws1}--\ref{prop:ws2} when using two simple strategies for the block selection.
\end{remark}

\section{Numerical experiments}\label{sec:num}
In this section, we report some numerical results.
The experiments were run in Matlab R2024a on an Apple MacBook Pro with an Apple M1 Pro Chip and 16 GB RAM.

Given a set of samples $\{a^1,\ldots,a^m\} \subseteq \R^n$ and labels $\{b^1,\ldots,b^m\} \subseteq \R$,
let $\varphi_x \colon \Rn \to \R$ be a prediction function parameterized by a vector $x$.
We consider optimization problems from regression and classification models
where the objective function has the following form: 
\begin{equation}\label{ml_problem}
f(x) = \frac 1m \sum_{i=1}^m \ell(b^i,\varphi_{x}(a^i)) + \lambda P(x),
\end{equation}
with $\ell \colon \R \times \R \to [0,\infty)$ being a loss function, $P \colon \Rn \to [0,\infty)$ being a regularizer
and $\lambda \ge 0$ being a regularization parameter.

In what follows, we consider two types of comparison:
\begin{itemize}
\item In Subsection~\ref{subsec:num_upd}, we compare \IBCN\ with other schemes that use the same greedy block selection but different block updates;
\item In Subsection~\ref{subsec:num_sel}, we compare \IBCN\ with other schemes that use the same cubic Newton updates but different block selection.
\end{itemize}

\subsection{Comparison with other block updates}\label{subsec:num_upd}
We first compare \IBCN\ with two block coordinate descent methods using
the same greedy selection rule.
Specifically, we consider both a first-order method and a second-order method,
referred to as \BCDFIRST\ and \BCDSECOND, respectively.
At each iteration $k$ of \BCDFIRST\ and \BCDSECOND, given the current point $x_k$ and a block of variables $\I_k$,
we compute a search direction $d_k \in \RIk$ as follows:
\begin{itemize}
\item For \BCDFIRST, we use the steepest descent direction, that is,
    \[
    d_k = -\nabla_{\I_k} f(x_k);
    \]
\item For \BCDSECOND, we use a diagonally scaled steepest descent direction~\cite{bertsekas:1999,tseng:2009}, that is,
    \[
    d_k = -(H_k )^{-1} \nabla_{\I_k} f(x_k),
    \]
    where $H_k \in \R^{\abs{\I_k} \times \abs{\I_k}}$ is symmetric and positive definite.
    To compute $H_k$, we choose a diagonal Hessian approximation as in~\cite[Subsection 7.2]{tseng:2009}, that is,
    \[
    H_k = \text{diag}(v_k), \quad \text{with} \quad v_k = \bigl[ \min\{\max\{\nabla_{\{j\}}^2 f(x_k),10^{-2}\},10^9\}\bigr]_{j \in \I_k},
    \]
    where $\text{diag}(v_k)$ denotes the diagonal matrix constructed from the vector $v_k$.
\end{itemize}
For both \BCDFIRST\ and \BCDSECOND, once $d_k$ is obtained, we set $x_{k+1} = x_k + \alpha_k U_{\I_k} d_k$, with $\alpha_k$ being computed by means of an Armijo line search, similarly as in~\cite{bertsekas:1999,tseng:2009}.

At each iteration $k$ of \IBCN, \BCDFIRST\ and \BCDSECOND, a block of variables $\I_k$ is chosen
as described in Proposition~\ref{prop:ws1}, that is, such that $\norm{\nabla_{\I_k} f(x_k)}_{\infty} = \norm{\nabla f(x_k)}_{\infty}$.
In more detail, first we compute the index $\hat \imath_k$ corresponding to the largest component in absolute value of $\nabla f(x_k)$,
then $\I_k$ is set to include $\hat \imath_k$ with the other variable indices being chosen randomly.
In our experiments, we use blocks of size $q \in \{1,5,10,20,50,100\}$.

In \IBCN, we set $\sigma_0 = \sigmamin = 1$, $\eta_1 = \eta_2 = 0.1$, $\gamma_1 = 1$, $\gamma_2 = \gamma_3 = 2$ and $\tau = 1$.
To compute $s_k$ at each iteration $k$ of \IBCN,
we set $s_k = \hat s_k$, with $\hat s_k$ defined as in~\eqref{sk_m_block}, if this choice satisfies~\eqref{sk_gm_block}.
Otherwise, we run a Barzilai-Borwein gradient method~\cite{raydan:1997} to $m_k(s)$, starting from $\hat s_k$, until a point $s$ is produced
such that~\eqref{sk_gm_block} holds with $s_k$ replaced by $s$.

In all experiments,
we run \IBCN, \BCDFIRST\ and \BCDSECOND\ from the starting point $x_0 = 0$ for $10^4$ iterations without using any other stopping condition.
Then, considering a sequence $\{x_k\}$ produced by a given algorithm, we analyze the decrease of the objective error
$(f(x_k)-f^*)$ and the decrease of the first-order stationarity violation $\norm{\nabla f(x_k)}$,
with $f^*$ denoting the best objective value found for a given problem over all simulations.

Both non-convex and convex problems are considered in our experiments in Subsubsection~\ref{subsub:sp_ls} and \ref{subsec:l2_log_reg}, respectively.

\subsubsection{Sparse least squares}\label{subsub:sp_ls}
The problem of recovering sparse vectors from linear measurements is central in many applications,
such as compressive sensing~\cite{donoho:2006} and variable selection~\cite{fan:2001}.
To obtain sparse solutions, a popular approach is to use least squares with
$\ell_1$-norm regularization, resulting in a convex formulation known as LASSO~\cite{tibshirani:1996}.
However, in order to overcome the bias connected to the $\ell_1$ norm, some non-convex regularizers have also been introduced in the literature~\cite{wen:2018}.

Here we use a non-convex sparsity promoting term considered in, e.g., \cite{lai:2013,pant:2014},
given by $P(x) = \sum_{i=1}^n (x_i^2 + \omega^2)^{p/2}$, with small $\omega>0$ and $p \in (0,1)$.
Using the least squares as loss, we hence obtain the following non-convex problem:
\begin{equation}\label{sp_ls_problem}
\min_{x \in \Rn} \frac 1m \norm{Ax - b}^2 + \lambda \sum_{i=1}^n (x_i^2 + \omega^2)^{p/2},
\end{equation}
where $A = \begin{bmatrix} a^1 & \ldots & a^m \end{bmatrix}^T \in \R^{m \times n}$ and $b = \begin{bmatrix} b^1 & \ldots & b^m \end{bmatrix}^T \in \R^m$.
In our experiments we set $\lambda = 10^{-2}$, $\omega = 10^{-2}$ and $p = 0.5$.
After generating the elements of the matrix $A$ randomly from a uniform distribution in $(0,1)$, with $m = n = 10,000$,
a vector $\hat x \in \Rn$ was created with all components equal to zero except for $1\%$ of them, which were randomly set to $\pm 1$.
Then, we set $b = A \hat x + \zeta$, where $\zeta \in \R^m$ is a noise vector with elements drawn from a normal distribution
with mean $0$ and standard deviation $10^{-3}$.

Note that, according to the notation used in~\eqref{ml_problem}--\eqref{sp_ls_problem}, we have
\[
\nabla f(x) = \frac 2m A^T (Ax - b) + \lambda \nabla P(x), \quad \text{and} \quad
\nabla^2 f(x) = \frac 2m A^T A + \lambda \nabla^2 P(x).
\]
In the above formulas, the most computationally expensive operations are those involving the quadratic term, as the regularizer $P(x)$ has a separable structure leading to a negligible cost.
In particular, when the problem dimension is large, the computation of $A^T A$ can be prohibitive, as it requires $\bigO{mn^2}$ arithmetic operations, and should therefore be avoided.
Considering such a scenario, here we do not store $A^T A$ in the considered algorithms. Consequently, at every iteration, to obtain the Hessian of $f$ with respect to a block of size $q$, we compute inner products between selected columns of $A$, with a cost of $\bigO{mq^2}$ arithmetic operations, except when using blocks of dimension $q=1$, in which case we store the diagonal of $A^T A$ since this only requires an affordable $\bigO{mn}$ cost in the initialization.
Then, to get first-order derivatives at iteration $k$, we use the residual vector $r_k := Ax_k - b$, so that
\begin{equation}\label{gk_residual}
\nabla f(x_k) = \frac 2m A^T r_k + \lambda \nabla P(x_k).
\end{equation}
In this way, the computation of $\nabla f(x_k)$ requires $\bigO{mn}$ arithmetic operations, 
while the cost to update $r_{k+1}$ from $r_k$ is $\bigO{mq}$, since 
$r_{k+1} = r_k + A U_{\I_k} s_k$ with $U_{\I_k} s_k$ having at most $q$ non-zeros.
This strategy is commonly adopted when dealing with	quadratic functions in a block-coordinate setting (see, e.g., \cite{cristofari:2019a,hsieh:2008,necoara:2014}).

Overall, we see that the cost to compute the required first- and second-order derivatives
at every iteration is $\bigO{m\max\{q^2,n\}}$ arithmetic operations, which in our case reduces to $\bigO{mn}$ since $n \ge q^2$.

We run $10$ simulations and in Figure~\ref{fig:ls_sp} we report the average results with respect to both the number of iterations and the CPU time.

We see that, for $q = 1$, all the considered methods perform very similarly,
whereas \IBCN\ clearly outperforms both \BCDFIRST\ and \BCDSECOND\ as the size of the blocks increases.
In particular, while the results of \IBCN\ and \BCDSECOND\ remain comparable for $q = 5$,
we see that, for $q \ge 10$, \IBCN\ yields a much faster decrease in both the objective function and the gradient norm than the competing methods.
Within the given limit of $10^4$ iterations, we also observe that \IBCN\ is consistently able to achieve a lower objective value with a smaller gradient norm.

\begin{figure}
\centering
\subfloat[Objective error vs iteration]
{\includegraphics[scale=0.8, trim = 0.4cm 2cm 0.65cm 0.3cm, clip]{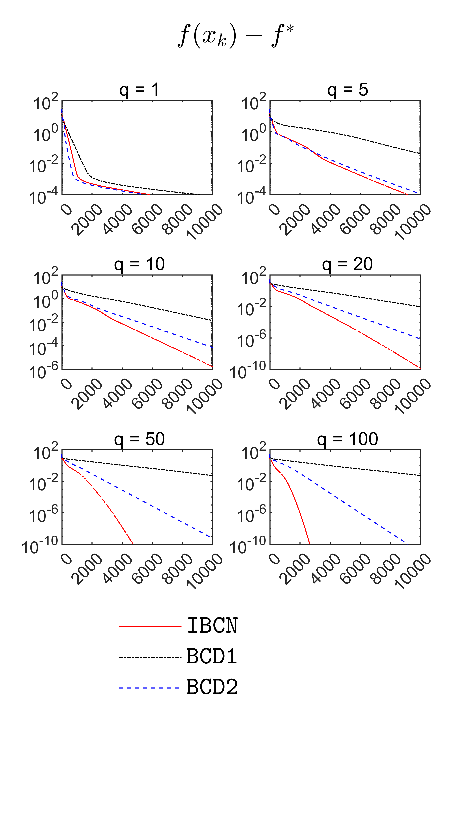}} \qquad \qquad
\subfloat[Stationarity violation vs iteration]
{\includegraphics[scale=0.8, trim = 0.4cm 2cm 0.65cm 0.3cm, clip]{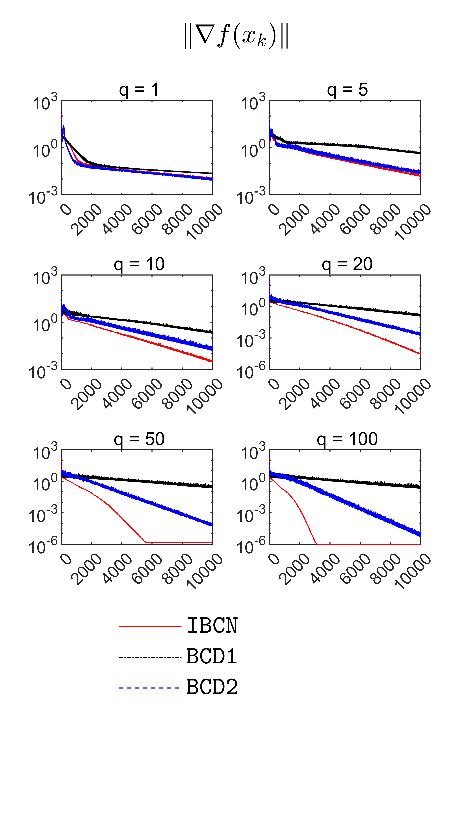}} \\
\subfloat[Objective error vs CPU time]
{\includegraphics[scale=0.8, trim = 0.4cm 2cm 0.55cm 0.3cm, clip]{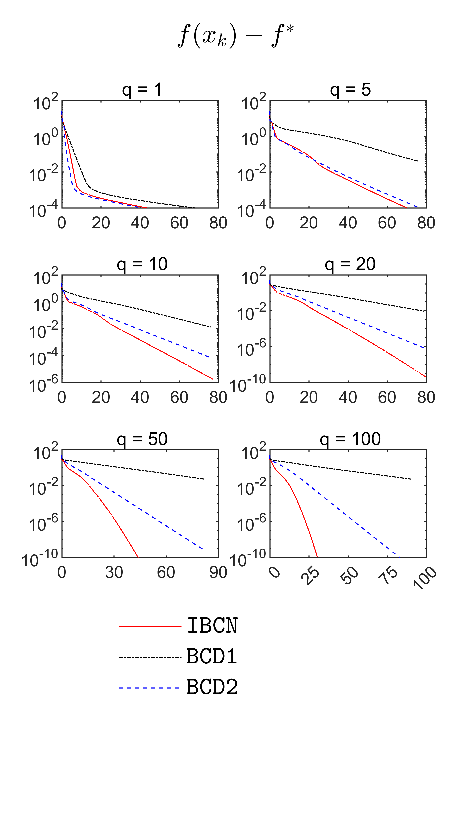}} \qquad \qquad
\subfloat[Stationarity violation vs CPU time]
{\includegraphics[scale=0.8, trim = 0.4cm 2cm 0.55cm 0.3cm, clip]{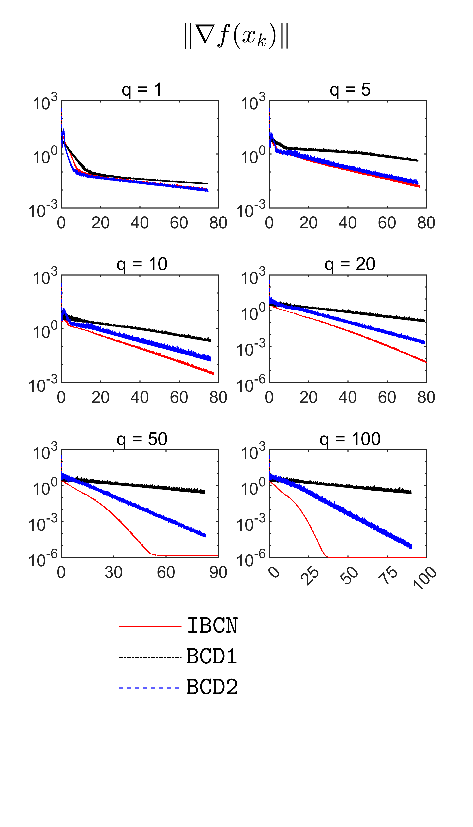}}
\caption{Results on sparse least squares using blocks of size $q$. In each plot, the $y$ axis is in logarithmic scale.}
\label{fig:ls_sp}
\end{figure}

\subsubsection{Regularized logistic regression}\label{subsec:l2_log_reg}
To assess how \IBCN\ works on convex problems, we consider the $\ell_2$-regularized logistic regression.
In particular, assuming that $b^i \in \{\pm 1\}$, $i = 1,\ldots,m$, the optimization problem can be formulated as follows:
\[
\min_{(x,z) \in \R^{n+1}} \frac 1m \sum_{i=1}^m \log \Bigl(1 + e^{-b^i((a^i)^T x + z)}\Bigr) + \lambda \norm x^2,
\]

We use the following three datasets from \url{https://www.csie.ntu.edu.tw/~cjlin/libsvmtools/datasets/}:
\begin{enumerate}
\item[(i)] gisette (train), $m = 6000$, $n = 5000$;
\item[(ii)] leu (train), $m = 38$, $n = 7129$;
\item[(iii)] madelon (train), $m = 2000$, $n = 500$;
\end{enumerate}
scaling all features of the last dataset in $[-1,1]$, while the other ones had already been scaled or normalized.

Results with respect to the number of iterations are reported in Figures~\ref{fig:l2_log_reg_f_it}--\ref{fig:l2_log_reg_g_it}.
We see that, for $q=1$, \IBCN\ and \BCDSECOND\ have similar performance and both of them give better results than \BCDFIRST.
For larger values of $q$, \IBCN\ provides a faster objective decrease and is able to
produce points with a smaller gradient norm than the two competitive methods.
As in the case of non-convex problems discussed in the previous subsection,
the superiority of \IBCN\ becomes more evident as the block size increases.

Finally, results with respect to the CPU time are reported in Figure~\ref{fig:l2_log_reg_gisette_time} only for the gisette dataset
since, for the other datasets, the methods take a few seconds in most cases.
We see that \IBCN\ seems to provide the best results for $q \ge 5$, confirming the above findings.

\begin{sidewaysfigure}
\centering
\subfloat[gisette dataset]
{\includegraphics[scale=0.8, trim = 0.4cm 2cm 0cm 0.3cm, clip]{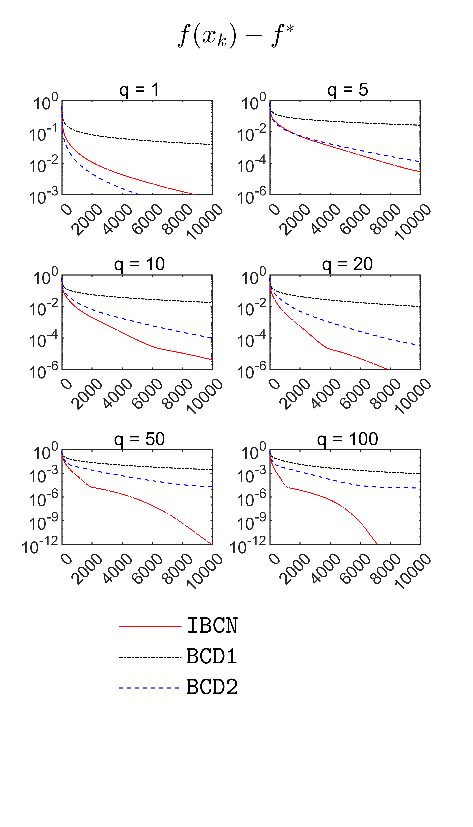}} \;
\subfloat[leu dataset]
{\includegraphics[scale=0.8, trim = 0cm 2cm 0cm 0.3cm, clip]{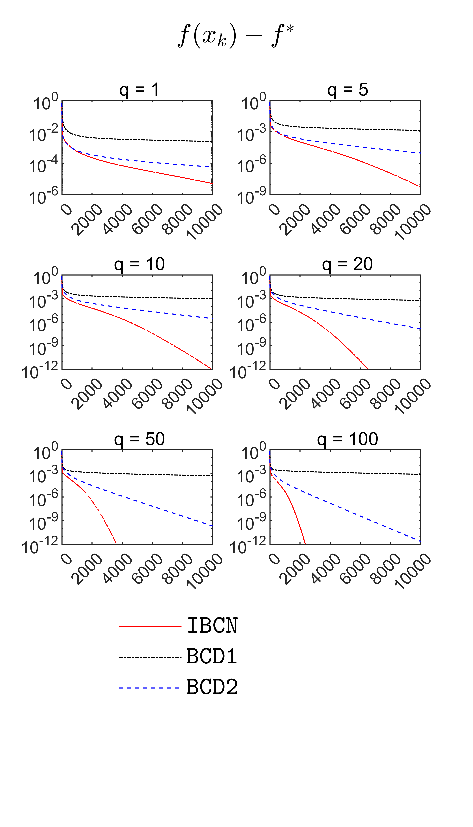}} \;
\subfloat[madelon dataset]
{\includegraphics[scale=0.8, trim = 0cm 2cm 0.65cm 0.3cm, clip]{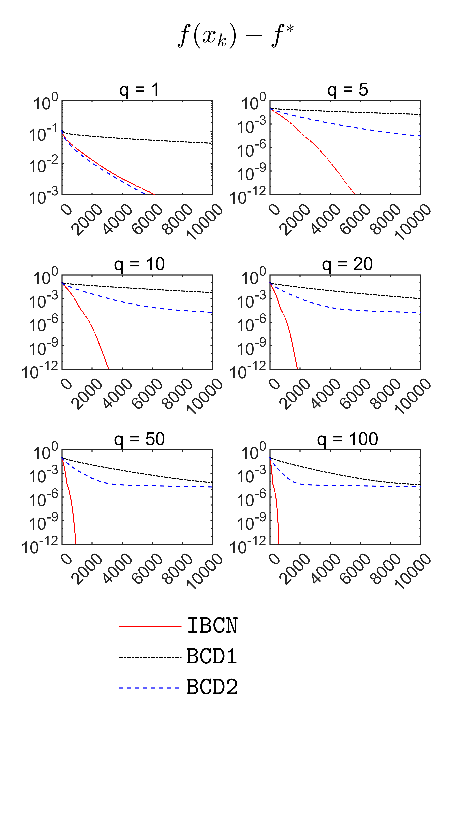}}
\caption{Objective error vs iteration for $\ell_2$-regularized logistic regression using blocks of size $q$. In each plot, the $y$ axis is in logarithmic scale.}
\label{fig:l2_log_reg_f_it}
\end{sidewaysfigure}

\begin{sidewaysfigure}
\centering
\subfloat[gisette dataset]
{\includegraphics[scale=0.8, trim = 0.4cm 2cm 0cm 0.3cm, clip]{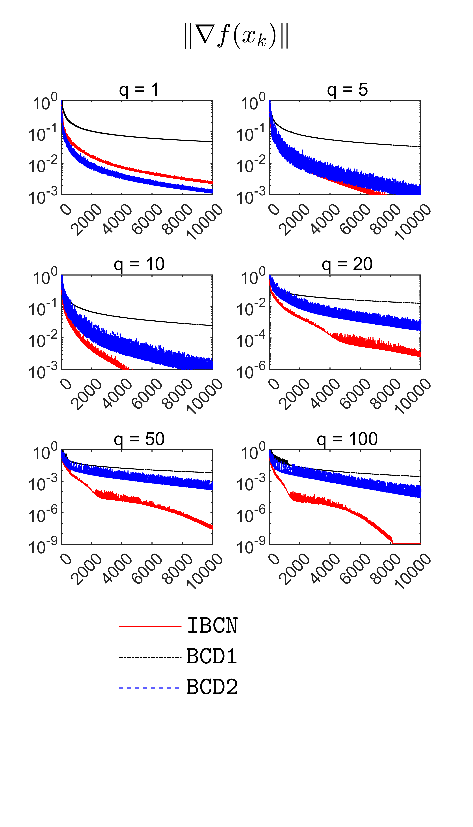}} \;
\subfloat[leu dataset]
{\includegraphics[scale=0.8, trim = 0cm 2cm 0cm 0.3cm, clip]{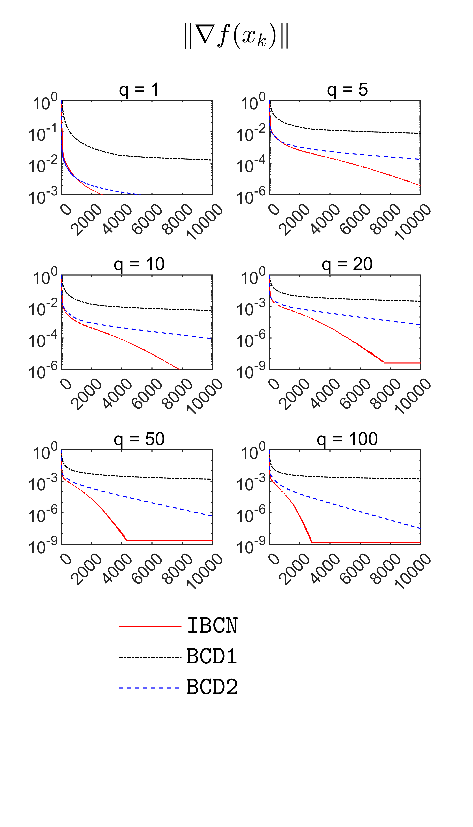}} \;
\subfloat[madelon dataset]
{\includegraphics[scale=0.8, trim = 0cm 2cm 0.65cm 0.3cm, clip]{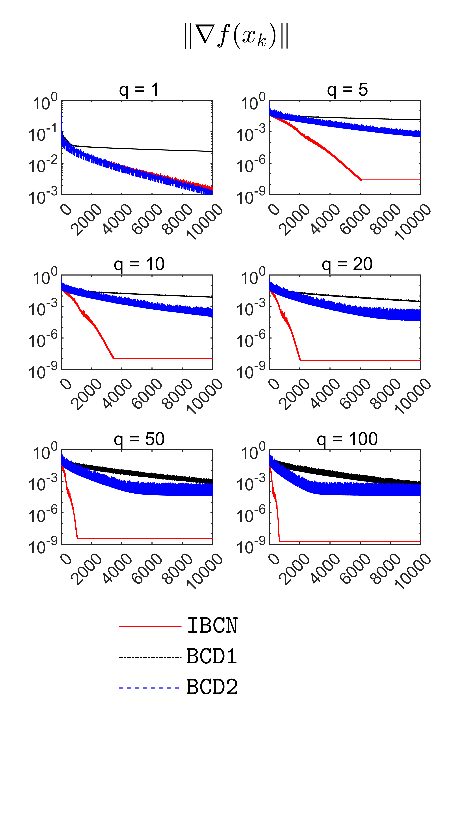}}
\caption{Stationarity violation vs iteration for $\ell_2$-regularized logistic regression using blocks of size $q$. In each plot, the $y$ axis is in logarithmic scale.}
\label{fig:l2_log_reg_g_it}
\end{sidewaysfigure}

\begin{figure}
\centering
\subfloat[Objective error vs CPU time]
{\includegraphics[scale=0.8, trim = 0.4cm 2cm 0.55cm 0.3cm, clip]{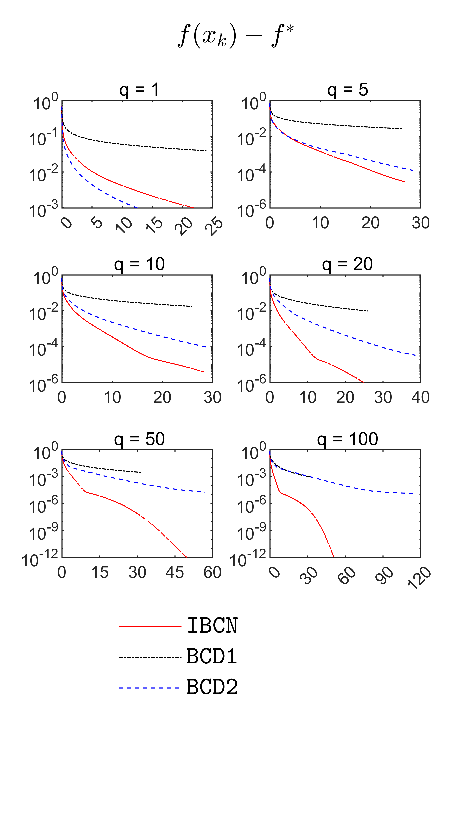}} \qquad \qquad
\subfloat[Stationarity violation vs CPU time]
{\includegraphics[scale=0.8, trim = 0.4cm 2cm 0.55cm 0.3cm, clip]{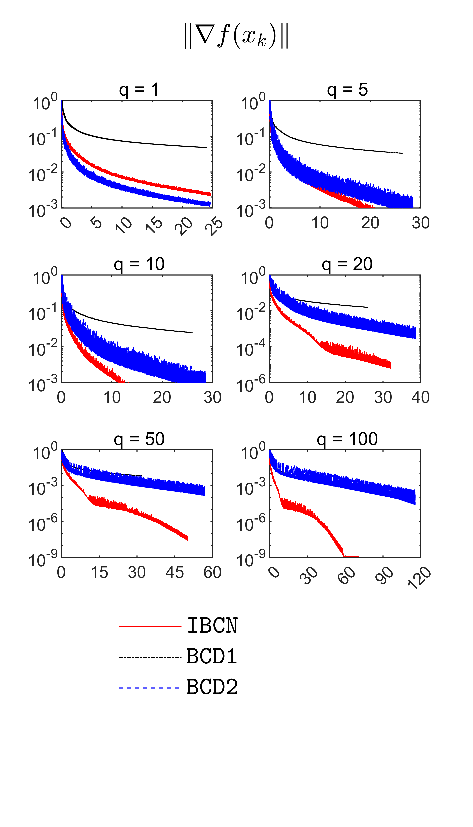}}
\caption{Results on $\ell_2$-regularized logistic regression with respect to the CPU time using blocks of size $q$ for gisette dataset.
In each plot, the $y$ axis is in logarithmic scale.}
\label{fig:l2_log_reg_gisette_time}
\end{figure}

\subsection{Comparison with other block selection rules}\label{subsec:num_sel}
As mentioned in Section~\ref{sec:intro}, different block selection rules can be used in a block coordinate descent scheme,
with the most efficient choice depending on computational aspects related to the features of the specific problem \cite{dhillon:2011,nutini:2015,nutini:2022,venturini:2023}.
Here, we want to compare the proposed greedy selection with two popular choices~\cite{nutini:2022,wright:2015}:
\begin{itemize}
\item \textit{cyclic selection}, in which each variable is selected at least once within any window of a pre-specified number of iterations;
\item \textit{random selection}, in which blocks are chosen randomly at every iteration.
\end{itemize}
On the one hand, both cyclic and random selection rules do not require to compute the full gradient of the objective function at each iteration, and can therefore be more computationally efficient in some cases. On the other hand, a greedy selection uses more information, so larger per-iteration progress is expected.

For first-order methods, it is well known that, in practice, the best choice among greedy, cyclic and random selection is related to the relative cost of performing 
the update of all blocks versus performing a full gradient iteration~\cite{nutini:2015}.
Some comparisons can be found in~\cite{dhillon:2011,nutini:2015,nutini:2022,venturini:2023}.

For cubic Newton schemes, an analysis of cyclic selection was carried out in~\cite{amaral:2022}, while random selection strategies were studied in~\cite{doikov:2018,hanzely:2020,zhao:2024}.
In what follows, we investigate how the performance, measured in terms of running time, of the considered block selection rules is affected by the per-iteration cost.
For this analysis, we consider the sparse least squares problem from Subsubsection~\ref{subsub:sp_ls} and distinguish two scenarios:
\begin{itemize}
\item[(i)] the per-iteration cost of greedy selection is higher than that of cyclic and random selection, see Subsubsection~\ref{subsub_hessian_not_stored};
\item[(ii)] the per-iteration cost of greedy selection is similar to that of cyclic and random selection, see Subsubsection~\ref{subsub_hessian_stored}.
\end{itemize}
As will be shown, which scenario occurs depends on whether or not the matrix $A^TA$ can be stored.
In particular, in our analysis, the per-iteration cost is approximated by the number of arithmetic operations to compute the required first- and second-order derivatives
of the quadratic term, which we assume dominates the workload of minimizing the cubic model while the operations on the regularizer $P(x)$ have a negligible cost due to its separable structure.

In the following experiments, cyclic and random selection were implemented using a uniform random sampling strategy without and with replacement, respectively. We left all other algorithmic choices as in \IBCN.
Moreover, the methods using cyclic and random selection were stopped after the time taken by the method using the greedy selection on the same instance. The results were then averaged over 10 simulations.

\subsubsection{Greedy selection: higher per-iteration cost than cyclic and random selection}\label{subsub_hessian_not_stored}
If the matrix $A^TA$ is not stored, in Subsubsection~\ref{subsub:sp_ls} we described how, for each iteration $k$, the Hessian of $f$ with respect to a block of size $q$ requires $\bigO{mq^2}$ arithmetic operations, while we can use the residual vector $r_k$ to compute $\nabla f(x_k)$ in $\bigO{mn}$ arithmetic operations by means of~\eqref{gk_residual}.
Hence, to get the required first- and second-order derivatives for an iteration with greedy selection,
we have a total cost of $\bigO{m\max\{q^2,n\}}$ arithmetic operations, which boils down to $\bigO{mn}$ since $n \ge q^2$ in our experiments.

Employing cyclic or random selection, at each iteration we still need $\bigO{mq^2}$ arithmetic operations to compute the Hessian of $f$ with respect to a block of size $q$,
while we only have to compute $q$ first-order partial derivatives, the latter having a cost $\bigO{mq}$ still using~\eqref{gk_residual} for the selected rows of $A^T$.
Hence, the total cost is of $\bigO{mq^2}$ arithmetic operations for each iteration, that is, computing the Hessian of $f$ with respect to a block is the dominant cost.
We conclude that, for the greedy selection, the per-iteration cost is about $n/q^2$ times higher than for cyclic and random selection.

The results are reported in Figure~\ref{fig:ls_sp_bs}(a). 
As expected, cyclic and random selection are faster in most cases due to their lower per-iteration cost, even if we observe that the greedy selection performs remarkably better when $q = 1$.

\subsubsection{Greedy selection: similar per-iteration cost to cyclic and random selection}\label{subsub_hessian_stored}
Assume now that the matrix $A^TA$ can be stored. First note that the cost to compute second-order derivatives is negligible in this case.
To compute first-order derivatives, we can update the gradient of $f$ from iteration $k$ to iteration $k+1$ by
\[
\nabla f(x_{k+1}) = \nabla f(x_k) + \frac 2m (A^T A) U_{\I_k} s_k + \lambda (\nabla P(x_{k+1})-\nabla P(x_k)) \quad \forall k \ge 0.
\]
Since $U_{\I_k} s_k$ has at most $q$ non-zeros components, it follows that $\nabla f(x_k)$ can be computed with a cost of $\bigO{qn}$ arithmetic operations.
The same procedure also applies to compute $q$ first-order partial derivatives when using cyclic or random selection,
implying that the cost is the same as computing the full gradient, that is, $\bigO{qn}$.
(Note that, since $m = n$ in our experiments, the cost of computing the required first-order derivatives for cyclic and random selection
is the same as the one obtained in the previous scenario where $A^T A$ was not stored.)

The results are reported in Figure~\ref{fig:ls_sp_bs}(b). 
We see that the greedy selection consistently outperforms both the cyclic and the random selection. 
The reason is that the per-iteration cost is comparable across the considered strategies, while the greedy selection uses more information by employing the full gradient of $f$ to select a block at each iteration.
We also observe that computing $A^TA$ takes approximately 3 seconds on average. Therefore, storing that matrix seems the most efficient option in the current configuration (cf. Figure~\ref{fig:ls_sp_bs}(a)).

\begin{figure}
\centering
\subfloat[The Hessian of the quadratic term is not stored. Here greedy selection has a higher per-iteration cost than cyclic and random selection.\label{fig:ls_sp_bs0}]
{\includegraphics[scale=0.8, trim = 0.4cm 2cm 0.65cm 0.3cm, clip]{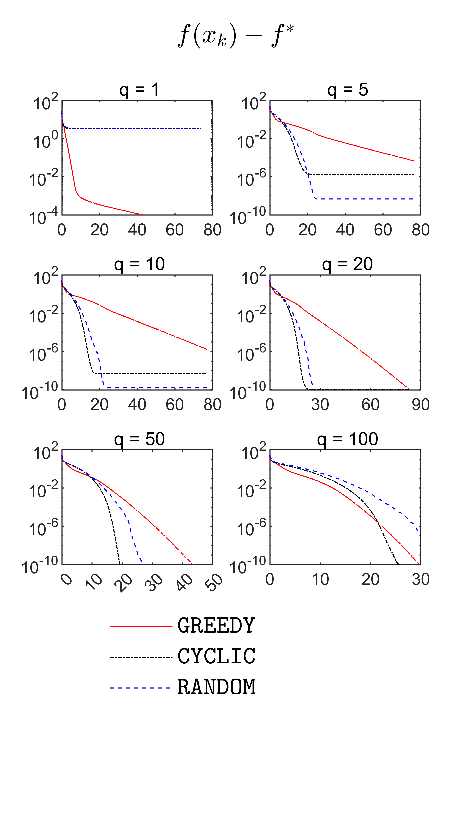}} \qquad \qquad
\subfloat[The Hessian of the quadratic term is stored. Here greedy selection has a similar per-iteration cost to cyclic and random selection.\label{fig:ls_sp_bs1}]
{\includegraphics[scale=0.8, trim = 0.4cm 2cm 0.65cm 0.3cm, clip]{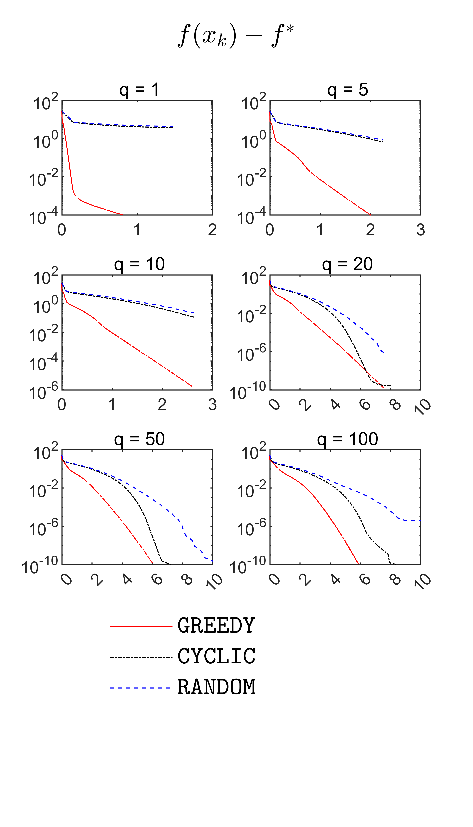}}
\caption{Objective error vs CPU time for sparse least squares using blocks of size $q$. In each plot, the $y$ axis is in logarithmic scale.}
\label{fig:ls_sp_bs}
\end{figure}

\section{Conclusions}\label{sec:concl}
In this paper, we have considered the unconstrained minimization of an objective function with Lipschitz continuous Hessian.
For this problem, we have presented a block coordinate descent version of cubic Newton methods using a greedy (Gauss-Southwell) selection rule,
where blocks of variables are chosen by considering the amount of first-order stationarity violation.
To update the selected block at each iteration, an inexact minimizer of a cubic model is computed.
In practice, such an inexact minimization can be carried out in finite time
without the need of additional evaluations of the objective function or its derivatives in other points.
In the proposed scheme, blocks are not required to have a predetermined structure and their size may even change during the iterations.
Moreover, the knowledge of the Lipschitz constant of the Hessian is not needed.

In a non-convex setting, we have shown global convergence to stationary points and analyzed the worst-case iteration complexity.
Specifically, we have shown that at most $\bigO{\epsilon^{-3/2}}$ iterations are needed to drive
the stationarity violation with respect to a selected block of variables below $\epsilon$,
while at most $\bigO{\epsilon^{-2}}$ iterations are needed to drive the stationarity violation with respect to all variables below $\epsilon$.
In particular, the latter result improves over $\bigO{\epsilon^{-3}}$ which was given in~\cite{amaral:2022} for cyclic-type selection.

Then, we have tested the proposed method, named \IBCN, on non-convex and convex problems arising in regression and classification models.

Numerical results indicate that the proposed method consistently outperforms other greedy schemes, with the performance gap widening as the block size increases. We have also shown that the proposed approach works particularly well in a regime where the full gradient computation does not significantly affect the per-iteration cost. In such a setting, the proposed greedy selection outperforms cyclic and random schemes.

The code of the proposed \IBCN\ method is freely available at \url{https://github.com/acristofari/ibcn}.

Finally, further investigation needs to be devoted to analyzing the worst-case iteration complexity in convex and strongly convex problems.

\section*{Disclosure statement}
No potential conflict of interest was reported by the author(s).

\newpage
\bibliographystyle{tfs}
\bibliography{cristofari2026}

\appendix
\section{Properties from Lipschitz continuity}\label{append:cubic}

\begin{proof}[Proof of Proposition~\ref{prop:lips_block}]
Choose $\I \subseteq \{1,\ldots,n\}$ and define the function $\psi \colon \Rn \to \RI$, $\psi(x) = U_{\I}^T \nabla f(x)$.
Namely, using~\eqref{subvec_g},
\[
\psi(x) = \nabla_{\I} f(x) \quad \forall x \in \Rn.
\]
Now, take $x \in \Rn$ and $s \in \RI$. Applying the mean value theorem to $\psi$, we can write
\[
\begin{split}
\nabla_{\I} f(x+U_{\I}s) - \nabla_{\I} f(x) & = \psi(x+U_{\I}s) - \psi(x) \\
                                            & = \int_0^1 \nabla \psi(x + t U_{\I} s)^T U_{\I} s \, dt \\
                                            & = \int_0^1 U_{\I}^T \nabla^2 f(x + t U_{\I} s) U_{\I} s \, dt \\
                                            & = \int_0^1 \nabla^2_{\I} f(x + t U_{\I} s) s \, dt,
\end{split}
\]
where we have used~\eqref{submat_h} in the last equality.
Adding $-\nabla^2_{\I} f(x) s$ to all terms, we obtain
\begin{equation}\label{upp_bound_lips_ineq1}
\begin{split}
\norm{\nabla_{\I} f(x+U_{\I}s) - \nabla_{\I} f(x) - \nabla^2_{\I} f(x) s} & = \norm[\Big]{\int_0^1 (\nabla^2_{\I} f(x + t U_{\I} s) - \nabla^2_{\I} f(x)) s \, dt} \\
                                                              & \le \int_0^1 \norm{(\nabla^2_{\I} f(x + t U_{\I} s) - \nabla^2_{\I} f(x)) s} \, dt \\
                                                              & \le \norm s \int_0^1 \norm{\nabla^2_{\I} f(x + t U_{\I} s) - \nabla^2_{\I} f(x)} \, dt \\
                                                              &  \le L_{\I} \norm s^2 \int_0^1 t \, dt \\
                                                              & = \frac{L_{\I}}2 \norm s^2,
\end{split}
\end{equation}
where the last inequality follows from~\eqref{hess_lips}.
Thus, \eqref{ineq_lips_block} holds.

To show~\eqref{ub_lips_block}, by the mean value theorem we can write
\begin{equation}\label{mvt_f}
\begin{split}
f(x+U_{\I}s) - f(x) = \int_0^1 \nabla f(x + t U_{\I} s)^T U_{\I} s\, dt = \int_0^1 \nabla_{\I} f(x + t U_{\I} s)^T s\, dt,
\end{split}
\end{equation}
where we have used~\eqref{subvec_g} in the last equality.
Adding $-\nabla_{\I} f(x)^T s - \frac12 s^T \nabla^2_{\I} f(x) s$ to all terms, we obtain
\[
\begin{split}
\abs[\Big]{f(x+U_{\I}s) - f(x) - \nabla_{\I} f(x)^T s - \frac12 s^T \nabla^2_{\I} f(x) s} & = \\
\abs[\Big]{\int_0^1 (\nabla_{\I} f(x + t U_{\I} s) - \nabla_{\I} f(x) - t \nabla^2_{\I} f(x)s)^T s\, dt} & \le \\
\int_0^1 \abs{(\nabla_{\I} f(x + t U_{\I} s) - \nabla_{\I} f(x) - t \nabla^2_{\I} f(x)s)^T s} dt & \le \\
\norm s \int_0^1 \norm{\nabla_{\I} f(x + t U_{\I} s) - \nabla_{\I} f(x) - t \nabla^2_{\I} f(x)s} dt & \le \\
\frac{L_{\I}}2 \norm s^3 \int_0^1 t^2\, dt & = \\
\frac{L_{\I}}6 \norm s^3, &
\end{split}
\]
where the last inequality follows from~\eqref{ineq_lips_block}.
Thus, \eqref{ub_lips_block} holds.
\end{proof}

\end{document}